\documentclass[a4paper]{amsart}
\usepackage{fixltx2e}
\usepackage[utf8]{inputenc}
\usepackage[T1]{fontenc}
\usepackage[charter]{mathdesign}
\usepackage{mathtools}
\usepackage[bookmarks=true,pdfdisplaydoctitle=true]{hyperref}
\hypersetup{
pdfpagemode=UseNone,
pdfstartview=FitH,
pdftitle={A uniform nilsequence Wiener--Wintner theorem for bilinear ergodic averages},
pdfauthor={Pavel Zorin-Kranich},
pdflang=en-US
}

\usepackage[safeinputenc=true,style=alphabetic,firstinits,maxnames=4,doi=false,eprint=true,isbn=false,url=false]{biblatex}

\newbibmacro{string+doiurl}[1]{%
  \iffieldundef{doi}{%
    \iffieldundef{url}{%
         #1%
      }{%
      \href{\thefield{url}}{#1}%
    }%
  }{%
    \href{http://dx.doi.org/\thefield{doi}}{#1}%
  }%
}
\DeclareFieldFormat[article,misc]{title}{\usebibmacro{string+doiurl}{\mkbibquote{#1}}}

\newcommand*{\mailto}[1]{\href{mailto:#1}{#1}}

\numberwithin{equation}{section}
\newtheorem{theorem}[equation]{Theorem}

\newtheorem{proposition}[equation]{Proposition}
\newtheorem{corollary}[equation]{Corollary}
\theoremstyle{definition}

\theoremstyle{remark}

\newcommand*{\Z}{\mathbb{Z}}

\def\<{\left\langle}
\def\>{\right\rangle}
\newcommand*{\inv}{^{-1}}

\newcommand{\aveN}{\frac{1}{N}\sum_{n=1}^N}
\newcommand{\aveFn}{\frac{1}{|\Phi_N|}\sum_{n\in \Phi_N}}

\newcommand*{\ol}[1]{\overline{#1}}

\newcommand*{\Gb}[1][]{G_{\bullet #1}}
\newcommand*{\poly}[1][\Gb]{P(\Z,#1)}

\DeclareMathOperator*{\Climinf}{C-lim\,inf}

\begin{document}
\title{A uniform nilsequence Wiener--Wintner theorem for bilinear ergodic averages}
\author{Pavel Zorin-Kranich}
\address{Universität Bonn\\
  Mathematisches Institut\\
  Endenicher Allee 60\\
  53115 Bonn\\
  Germany
}
\email{\mailto{pzorin@uni-bonn.de}}
\urladdr{\url{http://www.math.uni-bonn.de/people/pzorin/}}
\begin{abstract}
We show that a $k$-linear pointwise ergodic theorem on an ergodic measure-preserving system implies a uniform $k$-linear nilsequence Wiener--Wintner theorem on that system.
The assumption is known to hold for arbitrary systems and $k=2$ (due to Bourgain) and for distal systems and arbitrary $k$ (due to Huang, Shao, and Ye).
\end{abstract}
\maketitle

\section{Introduction}\label{sec:intro}
Let $(X,\mu,T)$ be an ergodic measure-preserving system and $\Phi$ a F\o{}lner sequence in $\Z$.
Call a sequence $(a_{n})_{n}$ a \emph{good weight for the $k$-linear pointwise ergodic theorem on $X$ along $\Phi$} if for some distinct non-zero integers $b_{1},\dots,b_{k}$ and any bounded functions $f_{1},\dots,f_{k}\in L^{\infty}(X)$ the limit
\[
\lim_{N\to\infty} \aveFn a_{n} \prod_{i=1}^{k} f_{i}(T^{b_{i}n}x)
\]
exists pointwise almost everywhere.
With this terminology, the nilsequence Wiener--Wintner theorem \cite[Theorem 2.22]{MR2544760} tells that nilsequences are good weights for the $1$-linear pointwise ergodic theorem on any measure-preserving system along the standard F\o{}lner sequence $\Phi_{N}=\{1,\dots,N\}$.
The $1$-linear pointwise ergodic theorem, that is, the fact that $a_{n}\equiv 1$ is a good weight, has been used as a black box in its proof.
Recently, Assani, Duncan, and Moore \cite{arXiv:1402.7094} showed that the sequences $a_{n}=e^{ip(n)}$, $p$ polynomial, are good weights for the $2$-linear pointwise ergodic theorem, similarly using the $a_{n}\equiv 1$ case as a black box.
Moreover, their result is uniform in the same way as the $1$-linear nilsequence Wiener--Wintner theorem in \cite{arxiv:1208.3977}.

In this note we prove the natural joint generalization of these results.
We fix a F\o{}lner sequence $\Phi$ and say that the system $(X,\mu,T)$ has property $P_{k}$ if $a_{n}\equiv 1$ is a good weight for the $k$-linear pointwise ergodic theorem on $X$ along $\Phi$.
It is a long-standing conjecture that every measure-preserving system satisfies $P_{k}$ for every $k$ along the standard F\o{}lner sequence $\Phi_{N}=\{1,\dots,N\}$, and we have nothing to add on this issue.
For $k=2$ this conjecture has been proved by Bourgain \cite{MR1037434} (see also \cite{MR2417419} and \cite{arxiv:1504.07134}).
Some partial results for $k>2$ can be found in \cite{arxiv:1406.5930} and \cite{MR1613556}.

Our main result is the following uniformity seminorm estimate.
We refer to the prequel \cite{arxiv:1208.3977} for definitions of various concepts related to nilmanifolds $G/\Gamma$ and to \cite{arxiv:1506.05748} for the definition of the modified uniformity seminorms $U^{k+l}(T,c)$.
\begin{theorem}
Suppose that the ergodic measure-preserving system $(X,\mu,T)$ has property $P_{k}$ with parameters $b_{1},\dots,b_{k}$ along the F\o{}lner sequence $\Phi$.
Then
\label{thm:main}
\begin{multline}
\label{eq:ave-uniform}
\int \limsup_{N\to\infty}
\sup_{G/\Gamma : \Gb \text{ has length } l} C_{G/\Gamma}^{-1}
\sup_{\substack{g\in\poly\\ F\in W^{r,2^{l}}(G/\Gamma)}}
\Big| \|F\|_{W^{r,2^{l}}(G/\Gamma)}\inv
\aveFn \prod_{i=1}^{k} f_{i}(T^{b_{i} n}x) F(g(n)\Gamma) \Big|^{2^{l+1}}\\
\lesssim_{b_{1},\dots,b_{k},l}
\min_{i} \|f_{i}\|_{U^{k+l}(T,c)}^{2^{l+1}},
\end{multline}
where $r = \sum_{m=1}^{l}(d_{m}-d_{m+1})\binom{l}{m-1}$ with $d_{i}=\dim G_{i}$, $c=c(b_{1},\dots,b_{k})$, and the positive constants $C_{G/\Gamma}$ depend only on the nilmanifold $G/\Gamma$, that is, the filtration $\Gb$, the lattice $\Gamma$, and the Mal'cev basis used to define the Sobolev spaces $W^{r,2^{l}}(G/\Gamma)$.
\end{theorem}
The case $k=2$, $l=1$ has been proved by Assani, Duncan, and Moore \cite{arXiv:1402.7094}, and the case of commutative $G$ by Assani and Moore \cite{arXiv:1408.3064,arXiv:1409.0463}.
An immediate consequence of Theorem~\ref{thm:main} is a nilsequence Wiener--Wintner theorem.
Before formulating it let us recall the following fact.
\begin{proposition}
\label{prop:good-weight-L2}
Let $(c_{n})_{n\in\Z}$ be a bounded complex-valued sequence.
Then the following statements are equivalent.
\begin{enumerate}
\item\label{gw:nil} For every nilsequence $(a_{n})$ the averages
\[
\aveN a_{n}c_{n}
\]
converge as $N\to\infty$.
\item\label{gw:poly} The sequence $(c_{n})$ is a good weight for polynomial multiple ergodic averages along $\{1,\dots,N\}$, i.e., for every measure-preserving system $(Y,\nu, S)$, integer polynomials $p_{1},\dots,p_{k}$, and functions $f_1,\ldots,f_k\in L^\infty(Y,\nu)$ the averages
\[
\aveN \phi(T^nx) S^{p_{1}(n)} f_1\cdots S^{p_{k}(n)} f_k
\]
converge in $L^2(Y,\nu)$ as $N\to \infty$.
\item\label{gw:lin} The sequence $(c_{n})$ is a good weight for linear multiple ergodic averages along $\{1,\dots,N\}$, i.e., \eqref{gw:poly} holds with $p_{i}(n)=b_{i}n$, $b_{i}$ arbitrary.
\end{enumerate}
\end{proposition}
\begin{proof}
The implication \eqref{gw:nil} $\implies$ \eqref{gw:poly} is \cite[Theorem 1.3]{MR2465660} and the implication \eqref{gw:poly} $\implies$ \eqref{gw:lin} is immediate.
Finally, the implication \eqref{gw:lin} $\implies$ \eqref{gw:nil} follows from \cite[Proposition 2.4]{arxiv:1407.0631}.
\end{proof}

We can now formulate our Wiener--Wintner theorem.
\begin{corollary}
\label{cor:WW}
Suppose that the ergodic measure-preserving system $(X,\mu,T)$ has property $P_{k}$ with parameters $b_{1},\dots,b_{k}$ along the standard F\o{}lner sequence $\Phi_{N}=\{1,\dots,N\}$.
Then for every functions $f_{1},\dots,f_{k}\in L^{\infty}(X,\mu)$ there exists a full measure set $X'\subset X$ such that for every $x\in X'$ the sequence
\[
(\prod_{i=1}^{k} f_{i}(T^{b_{i} n}x))_{n}
\]
satisfies the equivalent properties stated in Proposition \ref{prop:good-weight-L2}.
\end{corollary}
This follows from Theorem~\ref{thm:main} using the characterization \eqref{gw:nil} in Proposition \ref{prop:good-weight-L2} in a standard way (see e.g.\ \cite[\textsection 6]{arxiv:1208.3977} and use \cite[Lemma 2.6]{arxiv:1506.05748}).
By Bourgain's bilinear pointwise ergodic theorem Corollary~\ref{cor:WW} holds unconditionally for $k=2$.
In this case Corollary~\ref{cor:WW} has been previously proved in \cite[Theorem 1.4]{arXiv:1503.08863} and reproved in \cite{arxiv:1504.05732} after the appearance of this note.

\section{Proof of Theorem~\ref{thm:main}}
By induction on $l$.
For $l=0$ the group $G$ is trivial, so the nilsequences are constant and the averages over $n$ converge pointwise almost everywhere by property $P_{k}$.
Hence the left-hand side of \eqref{eq:ave-uniform} equals
\[
\Big\| \lim_{N\to\infty} \aveFn \prod_{i=1}^{k} f_{i}(T^{b_{i}n}x) \Big\|_{L^{2}_{x}}^{2},
\]
the limit now being taken in $L^{2}$.
The uniformity seminorm estimate for this limit originates in \cite[Theorem 12.1]{MR2150389}; the version used here can be found in \cite[Lemma 2.4]{arxiv:1506.05748}.

Suppose now that the claim holds for $l-1$.
Writing the function $F$ as a vertical Fourier series $F=\sum_{\chi}F_{\chi}$ as in \cite[(3.5)]{arxiv:1208.3977} we obtain for the supremum on the left-hand side of \eqref{eq:ave-uniform}
\begin{align*}
& \sup_{\substack{g\in\poly\\ F\in W^{r,2^{l}}(G/\Gamma)}}
\Big| \|F\|_{W^{r,2^{l}}(G/\Gamma)}\inv
\aveFn \prod_{i=1}^{k} f_{i}(T^{b_{i}n}x) F(g(n)\Gamma) \Big|^{2^{l+1}}\\
&\leq
\sup_{\substack{g\in\poly\\ F\in W^{r,2^{l}}(G/\Gamma)}}
\Big| \sum_{\chi} \frac{\|F_{\chi}\|_{W^{r-d_{l},2^{l}}(G/\Gamma)}}{\|F\|_{W^{r,2^{l}}(G/\Gamma)}} \|F_{\chi}\|_{W^{r-d_{l},2^{l}}(G/\Gamma)}\inv
\aveFn \prod_{i=1}^{k} f_{i}(T^{b_{i}n}x) F_{\chi}(g(n)\Gamma) \Big|^{2^{l+1}}\\
&\lesssim_{G/\Gamma}
\sup_{\substack{g\in\poly\\ \mathclap{F\in W^{r-d_{l},2^{l}}(G/\Gamma) \text{ vertical character}}}}
\Big| \|F\|_{W^{r-d_{l},2^{l}}(G/\Gamma)}\inv
\aveFn \prod_{i=1}^{k} f_{i}(T^{b_{i}n}x) F(g(n)\Gamma) \Big|^{2^{l+1}},
\end{align*}
where we have used \cite[Lemma 3.7]{arxiv:1208.3977} in the last line.
We incorporate the $G/\Gamma$-dependent constant into $C_{G/\Gamma}$, which may change from line to line.
Next we apply the van der Corput lemma \cite[Lemma 2.7]{arxiv:1208.3977} and estimate the integrand on the left-hand side of \eqref{eq:ave-uniform} by
\begin{multline*}
\Climinf_{m}
\limsup_{N\to\infty}
\sup_{G/\Gamma} C_{G/\Gamma}^{-1}
\sup_{\substack{g\in\poly\\ \mathclap{F\in W^{r-d_{l},2^{l}}(G/\Gamma) \text{ vertical character}}}}
\Big| \|F\|_{W^{r-d_{l},2^{l}}(G/\Gamma)}^{-2}
\aveFn \prod_{i=1}^{k} f_{i}(T^{b_{i}n}x) F(g(n)\Gamma)\\
\cdot
\prod_{i=1}^{k} \ol{f_{i}(T^{b_{i}(n+m)}x) F(g(n+m)\Gamma)} \Big|^{2^{l}},
\end{multline*}
where $\Climinf_{m} a_{m} := \liminf_{M\to\infty} \big| \frac{2}{M^{2}} \sum_{m=-M}^{M}(M-|m|) a_{m} \big|$.
With the notation for the cube construction from \cite[\textsection~3]{arxiv:1208.3977} this becomes
\[
\Climinf_{m}
\limsup_{N\to\infty}
\sup_{G/\Gamma} C_{G/\Gamma}^{-1}
\sup_{\substack{g\in\poly\\ \mathclap{F\in W^{r-d_{l},2^{l}}(G/\Gamma) \text{ vertical character}}}}
\Big| \|F\|_{W^{r-d_{l},2^{l}}(G/\Gamma)}^{-2}
\aveFn \prod_{i=1}^{k} \big( \ol{f_{i}} T^{b_{i}m}f_{i} \big)(T^{b_{i}n}x) \tilde F_{m}(\tilde g_{m}(n)\tilde \Gamma) \Big|^{2^{l}}.
\]
By \cite[Lemma 3.2]{arxiv:1208.3977} this is bounded by
\[
\Climinf_{m}
\limsup_{N\to\infty}
\sup_{G/\Gamma} C_{\tilde G/\tilde\Gamma}^{-1}
\sup_{\substack{\tilde g_{m}\in\poly[\tilde\Gb]\\ \mathclap{\tilde F_{m}\in W^{r-d_{l},2^{l-1}}(\tilde G/\tilde\Gamma)}}}
\Big| \|\tilde F_{m}\|_{W^{r-d_{l},2^{l-1}}(\tilde G/\tilde\Gamma)}^{-1}
\aveFn \prod_{i=1}^{k} \big( \ol{f_{i}} T^{b_{i}m}f_{i} \big)(T^{b_{i}n}x) \tilde F_{m}(\tilde g_{m}(n)\tilde \Gamma) \Big|^{2^{l}}.
\]
Integrating the last display over $X$ and applying Fatou's lemma and the inductive hypothesis we obtain
\[
\Climinf_{m} \min_{i} \| \ol{f_{i}} T^{b_{i}m}f_{i} \|_{U^{k+l-1}(T,c)}^{2^{l}}.
\]
By Hölder's inequality and the inductive construction of the Gowers--Host--Kra seminorms this is bounded by
\begin{align*}
&\min_{i} \Climinf_{m} \| \ol{f_{i}} T^{b_{i}m}f_{i} \|_{U^{k+l-1}(T,c)}^{2^{l}}\\
&\leq
\min_{i} \big( \Climinf_{m} \| \ol{f_{i}} T^{b_{i}m}f_{i} \|_{U^{k+l-1}(T,c)}^{2^{k+l-1}} \big)^{2^{-k+1}}\\
&\lesssim
\min_{i} \| f_{i} \|_{U^{k+l}(T,c)}^{2^{l+1}}
\end{align*}
as required.

\printbibliography
\end{document}
